\documentclass[11pt,fleqn]{article}

\usepackage{amscd, amsfonts, amsthm, amssymb, mathrsfs,
amsthm,bbm,makeidx} 

\usepackage{graphics,color}
\usepackage[all]{xy}
\usepackage{amssymb}
\usepackage{latexsym}
\usepackage{amsmath}
\numberwithin{equation}{section}

\newtheorem{Satz}{Satz}[section]
\newtheorem{theorem}[Satz]{Theorem}

\newtheorem{proposition}[Satz]{Proposition}

\newtheorem{remark}[Satz]{Remark}
\newtheorem{corollary}[Satz]{Corollary}
\newtheorem{lemma}[Satz]{Lemma}

\newcommand{\conf}{{\mathrm{conf}}}
\newcommand{\cone}{{\mathrm{cone}}}

\newcommand{\ra}{\rightarrow}

\newcommand{\DD}{{\cal{D}}}
\newcommand{\FFF}{{\cal{S}}}

\newcommand{\FF}{{\cal{F}}}

\newcommand{\RRR}{\mathbb{R}}
\newcommand{\NNN}{\mathbb{N}}

\newcommand{\CCC }{{\cal{C}}}

\newcommand{\DDD }{{\bf{\Delta}}}

\newcommand{\LLL}{{\mathrm{L}}}

\newcommand{\dom}{\mathop{\mathrm{dom}}}

\newcommand{\im}{\mathop{\mathrm{im}}}

\newcommand{\spec}{\mathrm{spec}}

\newcommand{\Lie}{{\cal{L}}}

\newcommand{\ZZZ}{{\mathbb{Z}}}

\newcommand{\Sing}{\mathrm{Sing}}

\newcommand{\C}{{\mathbb{C}}}

\def\haken{\mathbin{\vrule height0.6pt width0.6em \vrule height0.6em}}

\newcommand{\norm}[1]{\lVert #1 \rVert}

\begin{document}

\title{An analytic approach to the stratified Morse inequalities for complex cones}
\author{Ursula Ludwig}

\maketitle

\begin{abstract}
In a previous article the author extended the Witten deformation to singular spaces with cone-like singularities and to a class of Morse functions called admissible Morse functions. The method applies in particular to complex cones and stratified Morse functions in the sense of the theory developed by Goresky and MacPherson. It is well-known from stratified Morse theory that the singular points of the complex cone contribute to the stratified Morse inequalities in middle degree only.  In this article  an  analytic proof of this fact is given.


\end{abstract}

\section{Introduction}

For singular spaces the usual singular homology  looses the nice properties which hold for manifolds. In \cite{goresky2} Goresky and MacPherson introduced a new homology theory for a singular space $X$,  the so-called intersection homology $IH_*(X)$, which has all the nice properties one is used to have for singular homology on manifolds. In particular in \cite{goresky} Goresky and MacPherson developed  a stratified version of Morse theory on singular spaces.

Let us shortly recall one of the main results in \cite{goresky}, namely the stratified Morse inequalities. We restrict hereby to the situation where  $X$ is a compact, complex algebraic (or analytic) variety of dimension $\dim X = 2 \nu$ with isolated singularities, since this is the case treated further in this article.  Let us denote by $\Sing(X)$ the singular set of $X$ and let $f: X \rightarrow \RRR$ be a stratified Morse function on $X$. By definition of a stratified Morse function the restriction $f \restriction _ {X \setminus \Sing (X)}$ is a smooth Morse function and we denote by  $c_i(f \restriction _{X \setminus \Sing (X) })$  the number of critical points of $f \restriction _{X \setminus \Sing (X)}$ of index $i$.  Let us denote by $Ib _i(X)$ the Betti numbers for the intersection homology of $X$ with middle perversity and by
\begin{equation} IP _X(x) := \sum_{i=0}^ {2\nu} Ib _i (X) x^ i  \end{equation}  the Poincar\'e polynomial for the intersection homology of  $X$. The stratified Morse inequalities in Part II of \cite{goresky}, specialised to the situation here read as follows

\begin{theorem}\label{theoremgm} (\cite{goresky}, Section 6.12) 
There exists a polynomial  $Q(x) \in \ZZZ[x]$ with non negative coefficients such that \begin{equation}\label{MIgoresky} \sum_{i=0}^ {2\nu} c_i(f) x^i = IP_X(x) + (1+x) Q(x) ,\end{equation}
where $ \displaystyle \sum_{i=0}^ {2\nu} c_i(f) x^i$ is the Morse polynomial with coefficients \begin{equation}\label{morseinequ1}c_i (f) := \left \{ \begin{array}{ll}  c_i (f _{\restriction X \setminus \Sing (X)}) , &  i \not = \nu , \\
c_{\nu}(f_{\restriction X \setminus \Sing (X)}) +  \sum _{p \in \Sing (X)} m_p, & i = \nu. \\ \end{array} \right . \end{equation}  Hereby the contribution $m_p$ of a singular point $p \in \Sing (X)$ to the Morse inequalities is explicited below. It is concentrated in middle degree and is independent of the chosen stratified Morse function. 
\end{theorem}

Let us explain the contribution  $m_p$ in \eqref{morseinequ1}  in more detail:   Let $B_{\epsilon}(p)$ be the  closed $\epsilon$-ball around $p \in \Sing (X)$. The neighbourhood $B_{\epsilon}(p)$ is homeomorphic to $\cone _{[0,1]}(L_p):= ([0, 1] \times L_p )/ ( \{0\} \times L_p )$, where $L_p$ is a smooth manifold of dimension $2 \nu -1$ called the link of the singularity. Let us choose $0 < \delta << \epsilon$. Since $p$ is an isolated singular point of $X$ the local Morse data of $f$  at $p $ reduce in this case to the normal local Morse data.  They are defined as the pair of spaces   \begin{equation}\label{defmorsedata} (M_p, l_p^-) := \Big ( B_{\epsilon}(p) \cap f^{-1} ([f(p) - \delta, f(p) + \delta ]) , B_{\epsilon}(p) \cap f^{-1} ( f(p) - \delta ) \Big ) .\end{equation} The pair $(M_p, l_p^-)$ is independent of the  choice of $\delta$ and $\epsilon$, if $\delta$ and $\epsilon$ are chosen small enough and such that  $0 < \delta << \epsilon$.   The set $l_p^-$ is called the lower halflink of $f$ at $p$. It is not difficult to see that the  pair $(M_p, l_p^-)$ is homeomorphic to the pair $(B_{\epsilon } (p) , l^-_p)\simeq (\cone _{[0,1]}(L _p), l^-_p)$.  In the complex case the space of non degenerate covectors is connected (see \cite{goresky}, Proposition I.1.8) and therefore the homeomorphism type of the  Morse data near a singular point of $X$ is independent of the Morse function.  For the relative intersection homology with middle perversities of the pair $(M_p, l^ - _p)$ one has (see \cite{goresky})
\begin{equation} \label{morsedatum1} IH _i (M_p, l^ - _p) = 0 \textrm{ for } i \not = \nu .\end{equation} Moreover $m_p$ in \eqref{morseinequ1} is given by \begin{equation} m_p := \dim IH _{\nu} (M_p, l^ - _p)=\dim IH _{\nu} (\cone (L _{p} ), l^ - _p) , \end{equation} 
$\cone (L_p) := ([0, \infty) \times L_p)/(\{0\} \times L_p)$.
 As explained in \cite{goresky}  one can associate to this situation a Milnor fibration  whose fiber is a Stein space. The contribution $m_p$ can also be expressed in terms of the variation map for this fibration.  The  result \eqref{morsedatum1}  is related to  vanishing results for the homology of a Stein space (see \cite{goresky}, Part II, Chapter 6 for more details).

In \cite{ulconf} Morse theory on singular spaces has been approached  using the analytic method of the Witten deformation. The class of spaces  considered in \cite{ulconf} are conformally conic Riemannian manifolds in the sense defined in \cite{leschcone}. Those are a generalisation of  spaces with cone-like singularities. Their  intersection cohomology has an analytic expression in terms of the cohomology of the complex of $\LLL ^2$-forms.  The Witten deformation (proposed in \cite{witten}, rigorously proven in  \cite{hs4} for the smooth situation)  generalised to a conformally conic Riemannian manifold $X$  consists in deforming the complex of $\LLL^2$-forms using a certain class of functions, which were called admissible Morse functions in \cite{ulconf}. One gets again Morse inequalities, this time relating the $\LLL^2$-Poincar\'e polynomial to the number of critical points of the Morse function. Again the singular points of $X$ contribute to this Morse inequalities by $\dim IH^i (\cone (L_p), l_p^-)$ in degree $i$. However, in the analytic proof of the Morse inequalities in \cite{ulconf}, the contribution of the singularities to the Morse inequalities are computed in terms of the model Witten Laplacian $\DDD_t^{(i)}$
\begin{equation} \ker \DDD_t ^{(i)}\simeq  IH^i (\cone (L_p), l_p^-).\end{equation}
Note that in the general situation treated in \cite{ulconf} the contribution of the singularity to the Morse inequalities is not necessarily concentrated in middle degree only and in general depends on the Morse function. However if $X$ is a complex  cone  and $f:X \ra \RRR$ is a stratified Morse function, then $f$ is also admissible in the sense of \cite{ulconf} (see Proposition \ref{propmorsefunction}) and comparing the results in \cite{ulconf} and Theorem \ref{theoremgm} one gets that 

\begin{theorem}\label{main1} 
Let $X$ be a complex cone and $f: X \ra \RRR$ a stratified Morse function, then  \begin{equation} \label{vanishing}
  \ker \DDD_t ^{(i)}\simeq IH ^i (\cone (L_p), l_p^-) = 0 , \textrm{ for } i \not = \nu.
 \end{equation}

\end{theorem}

The goal of this paper is to give an analytic proof of Theorem \ref{main1}. We use the results in \cite{ulconf}, the Lefschetz Theorem for conformally conic K\"ahler manifolds (see \cite{leschcone}) together with a trick inspired from ``Gromov's trick'' in \cite{gromov} (Theorem 1.4.A).

Let us mention that the Witten deformation of a singular complex curve, which is a particular example of a complex cone, has been studied already in \cite{ul1} and \cite{ulcurve}.  

The paper is organised as follows: In Section 2 we recall basic facts on the $\LLL^2$-cohomology of a singular space and explain the Witten deformation on the complex of $\LLL^2$-forms. We also show that a stratified Morse function on a complex cone is admissible in the sense of \cite{ulconf}. In Section 3 we recall some of the main results in \cite{ulconf} which will be used in the proof of the main theorem, for convenience of the reader. In Section 4 we show some K\"ahler identities.  The main result is finally proved in Section 5.

\section{The Witten Deformation for Complex Cones and Stratified Morse Functions}\label{wittensection}

Let  $X$ be a  compact, complex cone of dimension $\dim X = 2 \nu$ and $f: X \ra \RRR$ a stratified Morse function in the sense of stratified Morse theory in \cite{goresky}. Then in particular $X$ with its metric induced from the Fubini-Study metric of an ambient projective space is a conformally conic Riemannian manifolds  in the sense of \cite{leschcone}. Let us denote by $\Sing (X)$ the singular set of $X$.

\begin{proposition}\label{propmorsefunction} Let  $X$ be a complex cone and $f: X \ra \RRR$ a stratified Morse function. Then

\begin{itemize} \item[(i)] The restriction $f\restriction_{ X \setminus \Sing(X)} $ is a smooth Morse function.  
\item[(ii)] Locally near a singular point $p \in \Sing(X)$ the function $f$ has the form \begin{equation} \label{morselemma} f(r, \varphi) = f(p) +  f_1 (r, \varphi) + f_2(r, \varphi) , \text{ where } f_1 = r h \text{ and } f_2 = O(r^{1+ \delta}),   \end{equation}
$h: L _p \ra \RRR$ is a smooth function and $r$ is the radial coordinate. Moreover there exists a neighbourhood $U$ of the singular set $\Sing(X)$ and a constant $a >0$ such that $\vert \nabla f \vert ^2 \geq a^2$ on $U$.

\end{itemize}
 
\end{proposition}

\begin{remark} A function on a conformally conic Riemannian manifold which satisfies (i) and (ii) was called  an admissible Morse function in \cite{ulconf}, Definition 2.1.  \end{remark}

\begin{proof} Condition (i) is clear from the definition of a stratified Morse function in \cite{goresky}. As in \cite{goresky}, Lemma II.2.1.4 we have  the following situation in a neighbourhood $U_p$ of a singular point of $X$:
\begin{equation}\label{up} U _p \cap X   \subset \C ^N  .\end{equation} The metric on $  U _p \cap X $ is induced from a K\"ahler metric on $\C ^ N$, and 
\begin{equation} f= \mathrm{Re}(F)_{\mid X} + O(r^{1+ \delta}) , \end{equation} 
where $F : \C ^N \ra \C$ is a linear function on the ambient space, $\mathrm{Re}(F)$ its real part. The local form in \eqref{morselemma} now follows easily, since $X$ is a cone near $p$. Condition (ii) follows from the non-degeneracy condition for a stratified Morse function: a stratified Morse function is not critical in normal directions.

\end{proof}

 In the rest of this section we  review basic facts on the $\LLL^2$-cohomology of conformally conic Riemannian manifolds (see \cite{cheeger2}, \cite{leschcone}).  We rephrase them   using  the language of Hilbert complexes as introduced in \cite{hilbert}. We also shortly  explain the Witten deformation as generalised in  \cite{ulconf}.  Let $(\Omega _0^*(X \setminus  \Sing(X) ), d)$ be the de Rham complex of differential forms with compact supports. The $\LLL^ 2 $-metric on forms is defined as \begin{equation} \langle \alpha  , \beta  \rangle = \int _{X \setminus \Sing(X)} \alpha \wedge * \beta .\end{equation}In the case of conformally conic manifolds of even dimension one has a unique {\it ideal boundary condition}, i.e.  the minimal and maximal extension of $d$ in the space of $\LLL^2$-forms coincide, 
\begin{equation}\label{dmindmax}
d_{i, \min  } = d _{i, \max} \text{ for all } i = 0, \ldots, 2\nu, 
\end{equation}
(see \cite{cheeger2}, \cite{leschcone}). We denote  by $(\CCC, d, \langle \ , \  \rangle)$ the unique extension of the differential complex $(\Omega _0^*(X \setminus \Sing(X) ), d)$ to a Hilbert complex. The cohomology of this complex is the so called $\LLL^2$-cohomology of $X$: \begin{equation} H_{(2)} ^i (X) := \ker d_{i, \min } / \im d_{i-1, \min }= \ker d_{i, \max } / \im d_{i-1, \max }.\end{equation}
Recall that since the Riemannian metric on $X$ is quasi-isometric to a cone-like one and $\LLL^2$-cohomology is an invariant of the quasi-isometry class, the $\LLL^2$-cohomology is isomorphic to the intersection cohomology of $X$ (see \cite{cheeger4}).

 The Witten deformation (\cite{witten}, \cite{hs4}) generalised to this singular setting consists in deforming the complex of $\LLL ^ 2$-forms using the stratified Morse function $f: X \ra \RRR$. Let us denote by $(\Omega ^*_0( X \setminus \Sing(X))  , d_t, \langle \ , \  \rangle)$  the differential complex of smooth forms with compact support on $X \setminus \Sing(X)$, where  $d_ t = e^{-ft}d e^{ft}$ and $\langle \ , \  \rangle$ is the $\LLL^2$-metric, $t \in ( 0, \infty)$ is the deformation parameter. 

By Proposition \ref{propmorsefunction} the stratified Morse function $f$ is admissible in the sense of  \cite{ulconf} and thus by the results shown there
 \begin{equation} \dom (d_{t, \max} )= \dom (d_{\max}) \text{ and } \dom (d _{t,\min}) = \dom (d_{\min}) \end{equation}
and therefore in view of \eqref{dmindmax}
\begin{equation} \dom (d_{t, \max} )= \dom ({d _{t, \min}}) .\end{equation}

 In other words  the  complex $(\Omega _0^*(X \setminus  \Sing(X)), d_t, \langle \ , \  \rangle)$ has a unique {\it ideal boundary condition}, and thus admits a unique extension into a Hilbert complex, which we denote by $(\CCC_t, d_t, \langle \ , \  \rangle)$. Let us denote by $\langle \ , \  \rangle _t$ the twisted $\LLL^ 2$-metric:
\begin{equation} \langle  \alpha , \beta  \rangle _t = \int _{X \setminus \Sing(X)}\alpha \wedge * \beta e^ {-2tf} .\end{equation} The map $\omega \mapsto e^ {-tf} \omega$ induces an isomorphism of Hilbert complexes between $(\CCC, d,  \langle \ , \  \rangle _t) $  and   $({\CCC _t}, {d_t}, \langle \ , \  \rangle) $. Thus also the deformed complex is Fredholm and its cohomology still computes the $\LLL^2$-cohomology of $X$.  Moreover the natural maps 
\begin{equation*} \ker \Delta _{t}^{(i)} =  \ker { d_{t, i}}\cap \ker {\delta_{t,i-1}} \longrightarrow H^i ({\CCC _t}, {d_t}, \langle \ , \  \rangle) \simeq H^i_{(2)}(X) , \quad i = 0, \ldots, 2 \nu, \end{equation*} are isomorphisms.   The  Laplacian $\Delta _t$ is the Laplacian associated to the Hilbert complex $(\CCC _t, d_t, \langle\ , \ \rangle )$  and is called the Witten Laplacian. Let us shortly recall that by  definition (of the Laplacian associated to a Hilbert complex) the Witten Laplacian  is   the closed self-adjoint non-negative extension of  $\Delta _{t} \restriction _{ \Omega _0^*}$ with domain:
\begin{equation} \label{domainlocalmodel}  
  \dom (\Delta  _t ) = \big \{ \omega  \  | \  \omega,  d_t \omega,  \delta_t \omega, d_t \delta _t \omega, \delta _t d_t \omega \in \LLL ^2 \big ( \Lambda ^*(T^*(X \setminus \Sing (X)) )\big ) \big \} . 
 \end{equation} 

Note moreover that  
\begin{equation}\label{domainfriedrich} \Delta _t^ {(i)} = \Delta _t^ {(i), \FF} \text{ for } i \not = \nu, \end{equation}
$\Delta _t^ {(i), \FF}$ being the Friedrichs extension of $\Delta _{t | \Omega _0 ^ *(X \setminus \Sing (X))}$ (see \cite{ulconf}, Proposition 2.6 (c)).

\section{The Spectral Gap Theorem and the Model Operator}

As in the smooth  situation   a  spectral gap theorem holds for the Witten Laplacian:

\begin{theorem}{\bf (Spectral gap theorem)}  \label{thmspectralgap} 
\begin{itemize}
\item[(i)] Let $X$ be a complex cone   and let $f:X \ra \RRR$ be a stratified  Morse function. Then there exist constants $C_1, C_2, C_3 >0$ and $t_0>0$ depending on $X$ and $f$ such that for any $t \geq t_0$,
\begin{equation*} \spec (\Delta _{t}) \cap (C_1e^{-C_2t}, C_3t) = \emptyset.\end{equation*}
\item[(ii)] Let us denote by $(\FFF_t, d_t, \langle \ , \  \rangle  )$ the subcomplex of $(\CCC _t, d_t,  \langle \ , \  \rangle )$ generated by all eigenforms of the Witten Laplacian $\Delta _t$ to eigenvalues in $[0,1]$.  Then, for $t \geq t_0$,
\begin{equation}\label{mpi1} \begin{array} {ll}  \dim \FFF^{i}_t =  c_i (f_{\mid X \setminus \Sing(X)}) + \displaystyle \sum _{ p \in \Sing(X)} m_p^i =:c_i(f),  \end{array} \end{equation} where the contribution of the singular point is given in terms of the intersection cohomology of the local Morse data (cf. \eqref{defmorsedata}), \begin{equation} m_p^i := \dim IH ^ i (M_p, l_p ^ -). \end{equation} 
\end{itemize}
   \end{theorem}

\begin{proof} By Proposition \ref{propmorsefunction}  any stratified Morse function on a complex cone is also admissible in the sense of \cite{ulconf}. It  has been proved in \cite{ulconf}, Theorem 1.1 that the spectral gap theorem holds for admissible Morse functions on conformally conic manifolds.  \end{proof}

From the spectral gap theorem one immediately deduces the following Morse inequalities (cf. Corollary 1.2 in \cite{ulconf}):

\begin{corollary}  \label{cormorse} In the situation of Theorem \ref{thmspectralgap}
\begin{equation}\label{morseinequ}
\begin{split} & \sum _{i = 0} ^ {k} (-1) ^{k-i}c_i (f)  \geq \sum _{i = 0} ^ {k} (-1) ^{k-i} b_i^{(2)} (X), \text{ for all }  0 \leq k <  2 \nu,
\\
&  \sum _{i = 0} ^ {2\nu} (-1) ^{i}c_i (f)  = \sum _{i = 0} ^ {2 \nu} (-1) ^{i} b_i^{(2)} (X), \end{split} \end{equation} where $b_i^{(2)} (X)$ denote the $\LLL^2$-Betti numbers of $X$.
\end{corollary}

As explained in the introduction  we can  compare the results in Theorem \ref{thmspectralgap} and Corollary \ref{cormorse} with the results from stratified Morse theory in \cite{goresky}. Using the duality between $\LLL^2$-cohomology and intersection homology, we deduce the claim of Theorem \ref{main1} immediately. However the goal here is to give an analytic proof of it.

The main step in the proof of the spectral gap theorem is the study of the local model operator near a singular point of $X$, which we shortly recall: Let $p \in \Sing(X)$.  As seen before a sufficiently small neighbourhood $U_p$ of $p$ is homeomorphic to $\cone(L_p)$, where $L_p$ is the link of the singularity. With no loss of generality we can assume that   the link $L_p$ of $X$ at $p$ is connected. In the following we denote by \begin{equation} \cone(L_p) := [0, \infty) \times L_p / \{0 \} \times L_p \end{equation}  the infinite cone over $L_p$. For $\epsilon > 0$ we denote by $\cone _{\epsilon }(L_p)$ the  open cone \begin{equation} \cone _{\epsilon }(L_p) := \{ (r, \varphi) \in \cone (L_p) \mid r < \epsilon \} .\end{equation}

We choose the metric $g_{\conf}$ on the infinite cone to be the metric of $X$ near the cone point and to be a conic model metric $dr^2 + r^2 g_{L_p}(0)$ for $r > 2 \epsilon$. Let $f: \cone (L) \ra \RRR$ be a stratified Morse function near the singularity.  Moreover for $r > 2 \epsilon$ let $f$ be of the form $ f= rh$ with $h :L \ra \RRR$ a function on the link. 

 For simplicity of notation we will from now on omit the subscript $_p$ and simply write $L$, $M$, $l^ -$ instead of $L_p$, $M_p$, $l^ -_p$ etc.

 We denote by $ \big ( \Omega ^*_0 (\cone(L)), d , \langle \ , \ \rangle \big )$ be the de Rham complex of smooth compactly supported forms on the infinite cone $ \big ( \cone(L) , g_{\conf} \big )$ and  by \linebreak $(\Omega ^ *_0 (\cone(L)), d_t, \langle \  , \   \rangle)$  the complex obtained by deforming the de Rham complex by using the function $f$, {\it i.e.} $d_t :=  e^{-tf} d e^{tf}$. Then there  is a unique Hilbert complex $(\DD _t, d_t, \langle \ , \ \rangle)$ extending the complex  $ \big ( \Omega ^ *_0 (\cone(L)), d_t, \langle \ , \  \rangle \big )$ (see \cite{ulconf}). The model Witten Laplacian $\DDD _t$  is defined as the Laplacian  associated to the Hilbert complex $(\DD _t, d_t, \langle \ , \ \rangle)$.   
\begin{theorem} \label{hodgelocal}
 \begin{itemize}
\item[(i)] The complex $(\DD_ t,d _t,\langle \ , \  \rangle) $ is Fredholm.  The natural maps 
\begin{equation}\label{hodgeisomlocal} \ker (\DDD _{t} ) \ra H^i(\DD _t, d_t, \langle \ , \ \rangle) , \quad i = 0, \ldots , 2 \nu ,\end{equation} 
are isomorphisms, where  $\DDD_t$ denotes the Laplacian associated to the complex $(\DD _t, d_t, \langle \ , \ \rangle)$ and $\DDD_t^{(i)}$ its restriction to $i$-forms. 
\item [(ii)] The model Witten Laplacian $\DDD_t$ satisfies a local spectral gap theorem: there exists $c>0$ such that for $t$ large enough \begin{equation} \spec (\DDD  _{t}) \subset \{0 \} \cup  [ct^2 , \infty). \end{equation} Moreover all forms in $\ker ( \DDD _{t})$, as well as their derivatives have  exponential decay outside a small neighbourhood of the singularity.
\item[(iii)] For the cohomology of the  complex $(\DD _t, d_t, \langle \ , \ \rangle)$ one gets \begin{equation} H^i(\DD _t, d_t, \langle \ , \ \rangle) \simeq IH ^i(\cone  (L), l^-) \simeq IH ^ i (M, l^ -),\end{equation} where $l^-$ is the lower halflink and $(M, l^ -)$ are the local Morse data, defined as  in \eqref{defmorsedata}. \end{itemize}

\end{theorem}

\begin{proof} Since by Proposition \ref{propmorsefunction} a stratified Morse function on a cone is admissible and therefore Theorem 3.1 in \cite{ulconf} yields the claims.  \end{proof}

\section{The Lefschetz Theorem and K\"ahler Identities}

The complex cone inherits a K\"ahler structure from the ambient projective space (equipped with the Fubini-Study metric). We denote by $J$ the complex structure  and by $\omega$ the K\"ahler form on $X$. Moreover we denote by $T_{\C} X$ (resp. $T^*_{\C} X$) the complexified tangent bundle (resp. cotangent bundle) on $X \setminus \Sing (X)$. Recall from \cite{leschcone}, Section 5 that the K\"ahler-Hodge Theorem as well as the $\LLL^2$-K\"ahler package holds for $X$.

\begin{theorem}(Lefschetz)\label{proplefschetz}

For $2k+2j= 2 \nu$ the map
\begin{equation}\begin{array}{ cccc} &  L^j : \LLL^2 ( \Lambda ^{k} (T^*_{\C}X ))& \longrightarrow &\LLL^2 (\Lambda ^{k +2j} (T^*_{\C}X)) \\
& \alpha & \mapsto &\omega ^j \wedge \alpha .\\ \end{array} \end{equation} is a (point-wise) quasi-isometric bijection, i.e. for each $\beta \in \LLL^2 (\Lambda ^{k+2j} )$ there exists a unique $ \alpha \in \LLL^2 (\Lambda ^{k} (T^*_{\C}X )$ such that
\begin{equation} L^j \alpha = \beta  \end{equation} and there is a constant (independent of $\beta$) such that 
\begin{equation}  C ^{-1} \norm{\alpha } \leq \norm{\beta}  \leq C \norm{\alpha} . \end{equation}

\end{theorem}

\begin{proof} See e.g. \cite{griffithsharris}, also \cite{gromov} (Theorem 1.2.A) and  \cite{leschcone} (Section 5).   \end{proof}

For the proof of the main result the following identities will be useful.

\begin{proposition}\label{lemlefschetzidentities}

\begin{itemize}
	\item[(i)] For $\alpha \in \dom (d_{t,\min}) = \dom (d_{t,\max})$ we have the following identity \begin{equation} [L, d_t] \alpha =0. \end{equation} 
	\item[(ii)]  Let $f:X \ra \RRR$ be the real part of a holomorphic function $F:X \ra \C$. For $j \in \NNN$ and  $ \alpha \in \Omega ^ *_{\C, 0} (X \setminus \Sing (X))$ we have the following identity:  \begin{equation}\label{kaehlerid} [L^j, \Delta _t] \alpha =0.\end{equation} 
\item[(iii)] Let $f:X \ra \RRR$ be as in (ii). Let $2k+2j= 2 \nu$.  Then  \begin{equation} [L^ j, \Delta _t^{(k)}]=0.\end{equation}  In particular $L^j (\dom (\Delta _t)) = \dom (\Delta _t)$
\end{itemize}

\end{proposition}

\begin{remark} The identity in (iii) will be used only locally near the singularities of $X$. \end{remark}

\begin{proof} (i) Recall that $[L, d_{\max}] = 0$ (see \cite{leschcone}, Section 5). Then also
\begin{equation} [L, d_t] \alpha  = [L, d] \alpha + t \omega  \wedge df \wedge \alpha - t df \wedge \omega  \wedge \alpha = 0 \end{equation} for all forms $\alpha \in \dom d_{t, \max} = \dom d_{\max}$.

(ii) It is enough to prove the claim for $j = 1$, the rest following by iteration. Recall (see e.g. \cite{griffithsharris})  that  on $\Omega ^ *_0 (X \setminus \Sing (X)))$ we have the K\"ahler identities \begin{equation} [L, \delta]= d^ c \text{ and } [L, \Delta ] =0. \end{equation}

 The Witten Laplacian can be expressed as follows (see e.g. \cite{bismutzhang}, Prop. 5.5) \begin{equation} \Delta _t = \Delta + t (\Lie _{\nabla f} + \Lie _{\nabla f}^ * ) + t^ 2 |\nabla f |^ 2, \end{equation} where $\Lie _{\nabla f}$ denotes the Lie derivative in the direction of the gradient vector field $\nabla f$ and $\Lie _{\nabla f}^ * $ denotes its adjoint.  Therefore to prove \eqref{kaehlerid}  it is enough to prove that  on $\Omega ^ *_0 (X \setminus \Sing (X))$:
\begin{equation} [L,\Lie _{\nabla f} + \Lie _{\nabla f}^ * ] =0. \end{equation}
Let us denote by $f': =\Im F$ the imaginary part of the holomorphic function $F$. Using  Cartan's identity and the K\"ahler identity $[L,d]=0$ an easy computation shows that 
\begin{equation} [L,\Lie _{\nabla f}] \alpha  =  - d (\nabla f \haken \omega)  \wedge \alpha = - d (d f' ) \wedge \alpha = - d^2f' \wedge \alpha = 0.\end{equation}
Similarly, using $[L, \delta ] = d^c$ one computes that 
\begin{equation} [L,\Lie _{\nabla f}^*] \alpha   = (d^cd f )\wedge \alpha =  0 ,\end{equation}
since $d^ cd$ is a real operator and $f$ is the real part of a holomorphic function.
The claim follows. 

(iii) Recall from  \eqref{domainfriedrich} that, for $l \not = \nu$, we have $\dom (\Delta _t^ {(l)} ) = \dom (\Delta _t^{(l), \FF})$, where  $\Delta ^ {\FF}_t$ denotes the Friedrichs extension of  $\Delta _{t| \Omega _0^*(X \setminus \Sing (X))}$. The  Friedrichs extension $\Delta _t^{ \FF}$ is the closure of $\Omega^*_0(X \setminus \Sing (X)) $ under the norm
\begin{equation} \sqrt{\langle \Delta _t\alpha, \alpha \rangle + \norm{\alpha}^ 2 }. \end{equation} Let $\alpha \in \dom (\Delta ^{(k+2j)}_t ) = \dom (\Delta ^{(k+2j) , \FF}_t )$. Then there exists a sequence $\alpha _n \in \Omega ^ {k+2j}_0(X \setminus \Sing (X))$ with 
\begin{equation} \begin{split} & \alpha _n \ra \alpha \text{ in } \LLL^2 , \text{ for } n \ra \infty \\
& \langle \Delta _t \alpha  _{nm}, \alpha_{nm} \rangle + \norm{\alpha _{nm}}^2 \ra 0, \text{ for } n\geq m \ra \infty ,\end{split} \end{equation}
where we denoted by $\alpha _{nm} := \alpha _n - \alpha _m$. By Theorem \ref{proplefschetz} there exists a unique form $\beta \in \LLL^2$ and a unique sequence $\beta _{n} \in \Omega _0^ k(X \setminus \Sing (X)) $ such that \begin{equation} \label{ljbeta} L^j \beta = \alpha, L^ j \beta _n = \alpha _n. \end{equation} 

Set again $\beta _{nm} : =\beta _n - \beta _m$. Using Theorem \ref{proplefschetz}, part (ii) and \eqref{ljbeta} we get that
\begin{equation}\begin{split}  |\langle  \Delta _t\beta _{nm} , \beta _{nm} \rangle| & \leq \norm{\Delta_t \beta _{nm} } \norm{\beta _{nm} } \leq C \norm{L^ j \Delta_t \beta _{nm} } \norm{L^j\beta _{nm} } \\ & =  C \norm{\Delta _t\alpha _{nm} } \norm{\alpha _{nm} } , \end{split} \end{equation} and thus  \begin{equation} \begin{split} & \beta _n \ra \beta \text{ in } \LLL^2 , \text{ for } n \ra \infty \\
& \langle \Delta _t \beta  _{nm}, \beta _{nm}\rangle + \norm{\beta _{nm}}^2 \ra 0, \text{ for } n\geq m \ra \infty .\end{split} \end{equation}

This shows that $\beta \in \dom (\Delta _t^ {(k) , \FF}) = \dom (\Delta _t^ {(k) })$ and  therefore \begin{equation} L^ j ( \dom (\Delta _t)) \supset \dom (\Delta _t). \end{equation} The inclusion $L^ j ( \dom (\Delta _t)) \subset \dom (\Delta _t)$ follows similarly. The rest of the claim follows using (ii).

 \end{proof}

\section{Proof of Theorem \ref{main1}}

We need one more result before we can prove  Theorem \ref{main1}.

\begin{lemma}\label{lembounded} There exists $\theta \in \LLL^2 (\Lambda ^1 (T^* \cone (L))) $ with 
\begin{itemize} \item[(i)]  $ \omega = d \theta$  and
\item[(ii)] $|\theta (r, \varphi) | < C $ on $\cone_{ \epsilon} (L)$  for some constant $ C >0$. \end{itemize}\end{lemma}
\begin{proof} The local situation near the cone point is as in \eqref{up}. The K\"ahler form $\tilde \omega$ on the ambient $\C ^N$ is closed and therefore exact, which gives (i). Moreover up to terms of order $2$ in $z$ one can write $\tilde \omega$ as $\tilde \omega = d \tilde {\theta}$, where  $\tilde {\theta} = \sum z_i d \overline{z_i}$, from which one deduces (ii).    \end{proof}

{\it Proof of  Theorem \ref{main1}}:  Recall that by Theorem \ref{hodgelocal} (i) and (iii),  we have
\begin{equation}  IH ^ k (M, l^ -) \simeq \ker \DDD _t ^ {(k)} \textrm{ for } k= 0, \ldots , 2 \nu.\end{equation}  We now prove that  
\begin{equation}\label{kert0} \ker \DDD _t ^ {(k)} = 0 \text{ for } k \not = \nu .\end{equation} We will prove \eqref{kert0} for the case where, near the singularities of $X$, the Morse function $f$ can be written as the real part of a holomorphic function. The general case then follows using in addition  perturbation arguments as in \cite{ulcurve}.

 The proof is inspired from \cite{gromov}, Theorem 1.4.A. 
Let $k > \nu$. Set $j :=  k- \nu$. Let us assume that there exists $0 \not = \alpha \in \ker \DDD _t^{(k)}$.  We can assume that $\norm{\alpha } = 1$. Let $\epsilon >0$, $\chi : \RRR ^+  \rightarrow  [0,1]$ be a cutoff function, with $\chi \restriction _{ [0, \epsilon /2] }= 1$, $\chi \restriction _{[ \epsilon , \infty)}  = 0$. Then by the Agmon estimates (Theorem \ref{hodgelocal} (ii)) we have for $t$ large enough \begin{equation}\label{dtdeltattruncated} |d_t (\chi \alpha )|= O(e^{-ct}) , \quad | \delta _t ( \chi \alpha) | = O(e^{-ct}) , \quad  \norm{ \chi \alpha} = 1+ O( e^{-ct}), \end{equation}
for some $c>0$.

By Theorem \ref{proplefschetz} there exists a unique $(2 \nu -k)$-form $\beta$ such that
\begin{equation}\label{followsfromlefschetz} \chi \alpha = L^j (\beta) = \omega ^j \wedge \beta , \quad  \norm{\beta} \leq C \norm{ \chi \alpha} .\end{equation}

By Proposition \ref{lemlefschetzidentities} (ii)  we get moreover that $\beta \in  \dom (\DDD _t^ {(2 \nu -k)})$ and together with \eqref{dtdeltattruncated} \begin{equation}\label{44} \norm{\DDD _t \beta } \leq  C\norm{\DDD _t (\chi \alpha)} =  O(e^{-ct})  .\end{equation} From \eqref{dtdeltattruncated}, \eqref{followsfromlefschetz} and \eqref{44}  we deduce that \begin{equation}\label{dtbeta0} \norm{ d_t \beta}^2  \leq \langle \DDD _t \beta , \beta \rangle \leq C \norm{\DDD _t \beta} \norm{\beta} =   O( e^{-ct}) . \end{equation}  Using Lemma \ref{lembounded} (i)  we may write
\begin{equation}\begin{split}  \chi \alpha &= \omega ^j \wedge \beta = d( \theta \wedge \omega ^{j-1} \wedge \beta) + \theta \wedge \omega ^{j-1} \wedge d \beta \\ & = d_t (\theta \wedge \omega^{ j-1} \wedge \beta) + \theta \wedge \omega ^{j-1} \wedge d_t \beta . \end{split} \end{equation}
Note that in particular, $\theta \wedge \omega ^{j-1} \wedge \beta \in \dom (d_t)$ locally near the cone point. Thus
\begin{equation} \begin{split} \label{chialpha} (\chi \alpha , \chi \alpha ) & = (\chi \alpha, d_t (\theta \wedge \omega^{ j-1} \wedge \beta)) + (\chi \alpha, \theta \wedge \omega^{ j-1} \wedge d_t \beta ) \\
& = (\delta _t (\chi \alpha) , \theta \wedge \omega ^{j-1} \wedge  \beta ) + (\chi \alpha, \theta \wedge \omega^{ j-1} \wedge d_t \beta )  . \end{split} \end{equation}

Using the Cauchy-Schwarz inequality, Lemma \ref{lembounded} (ii), \eqref{dtdeltattruncated} and \eqref{followsfromlefschetz} we get from  \eqref{chialpha}  that \begin{equation} 1 + O(e^{-ct}) = (\chi \alpha , \alpha )  \leq C ( \norm{\delta _t (\chi \alpha)} \norm{\beta} + \norm{\chi \alpha} \norm{  d_t \beta} ) \leq  O ( e^{-ct}) \end{equation} and  thus a contradiction. This shows that $\ker \DDD_t^k = 0 $ for $k > \nu$. 

 Let us denote by $\DDD_{ t, -} $ the model Witten Laplacian for the Morse function $-f$ and by $*$ the Hodge $*$-operator. Then \begin{equation} * \DDD _t^{(k)} \alpha = (-1) ^{k (2\nu-k)} \DDD _{t,-} ^{ (2 \nu-k)} * \alpha .\end{equation} Therefore \eqref{kert0}  for  $k < \nu$  now follows  by duality.  \hfill $\Box$

\paragraph{Acknowledgements}
I would like to thank Jean-Michel Bismut for suggesting work on the Witten deformation on singular spaces and for discussion. I also thank Jochen Br\"uning and 
Jean-Paul Brasselet for discussion. I would also like to thank my colleagues Marco K\"uhnel and Mario Listing for discussion.

\bibliographystyle{abbrv}

\end{document}